\documentclass[12pt,english,twoside,reqno,refpage]{amsart}
\usepackage[T1]{fontenc}
\usepackage[latin1]{inputenc}
\usepackage{amsthm}
\usepackage{amssymb}

\makeatletter
\numberwithin{equation}{section}
\numberwithin{figure}{section}
  \theoremstyle{plain}
  \newtheorem*{thm*}{\protect\theoremname}
\theoremstyle{plain}
\newtheorem{thm}{\protect\theoremname}[section]
  \theoremstyle{plain}
  \newtheorem{prop}[thm]{\protect\propositionname}

\makeatletter
\usepackage{a4wide}
\usepackage[all]{xy}
\usepackage{amssymb, amsmath, amsthm}
\usepackage{times}
\makeatother

\AtBeginDocument{

}

\makeatother

\usepackage{babel}
  \providecommand{\propositionname}{Proposition}
  \providecommand{\theoremname}{Theorem}
\providecommand{\theoremname}{Theorem}

\begin{document}

\title{Generating Geometry Axioms From Poset Axioms}

\author{Wolfram Retter}

\email{math.wolframretter(at)t-online.de}

\date{\textbf{\Large January 12, 2014}}
\begin{abstract}
Two axioms of order geoemtry are the poset axioms of transitivity
and antisymmetry of the relation 'is in front of' when looking from
a point. From these axioms, by looking from an interval instead of
a point, further well-known axioms of order geometry are generated
in the following sense: Transitivity when looking from an interval
is equivalent to \cite[§10, Assioma XIII]{peano_1889}. Assuming this
axiom, antisymmetry when looking from an interval is equivalent to
\cite[§1, VIII. Grundsatz]{pasch_1882}. Further equivalences, with
some of the implications well-known, are proved along the way.
\end{abstract}



\maketitle
\tableofcontents{}

\section{Introduction}

\noindent Let $X$ be a vector space over a totally ordered field
$K\,,$ for example $K=\mathbb{R}$ and $X=\mathbb{R}^{n}$ for an
$n\in\mathbb{Z}_{\geq1}\,.$ The \emph{vector interval relation} on
$X$ is the ternary relation $\left\langle \cdot,\,\cdot,\,\cdot\right\rangle $
defined by
\begin{align*}
\left\langle x,\, y,\, z\right\rangle  & :\Leftrightarrow\mbox{There is a }\lambda\in K\mbox{ such that }0\leq\lambda\leq1\mbox{ and }y=x+\lambda\left(z-x\right)\,,
\end{align*}
$X$ together with this relation satisfies the following conditions:
\begin{itemize}
\item For $a\in X\,,$ the binary relations $\left\langle \cdot,\,\cdot,\, a\right\rangle $
and $\left\langle a,\,\cdot,\,\cdot\right\rangle $ are reflexive
on $X\,.$ 
\item For $a\in X\,,$ the binary relation $\left\langle \cdot,\, a,\,\cdot\right\rangle $
is symmetric.
\item For $x,\, y\in X\,,$ $\left\langle x,\, y,\, x\right\rangle $ implies
$y=x\,.$
\end{itemize}
\noindent An \emph{interval space} is a pair consisting of a set $X$
and a ternary relation $\left\langle \cdot,\,\cdot,\,\cdot\right\rangle $
on $X$ such that these conditions are satisfied. Thus, a vector space
$X$ over a totally ordered field $K$ together with its vector interval
relation is an interval space. The concept of an interval space has
been taken from \cite[chapter I, 3.1]{verheul_1993}.

An interval space $\left(X,\,\left\langle \cdot,\,\cdot,\,\cdot\right\rangle \right)$
is also simply denoted by $X$ when it is clear from the context whether
the interval space or only the set is meant.

An interval space $X$ is called \emph{point-transitive} iff for each
$a\in X\,,$ the binary relation $\left\langle a,\,\cdot,\,\cdot\right\rangle $
is transitive, i.e. for all $x,\, y,\, z\in X\,,$ $\left(\left\langle a,\, x,\, y\right\rangle \mbox{ and }\left\langle a,\, y,\, z\right\rangle \right)\Longrightarrow\left\langle a,\, x,\, z\right\rangle \,.$
A vector space with its vector interval relation is point-transitive
and satisfies the equivalent conditions of the following theorem.
Condition (\ref{enu:interval_transitivity_criterion_1}) is obtained
from the point-transitivity condition that $\left\langle a,\,\cdot,\,\cdot\right\rangle $
is transitive by replacing the point $a$ by an interval. Condition
(\ref{enu:interval_transitivity_criterion_2}) is the interval relation
version of the strict interval relation condition \cite[§10, Assioma XIII]{peano_1889}.

~

\noindent 
\[
\xy<1cm,0cm>:(0,0)*=0{\bullet}="a",(0,-0.25)*=0{a},(2,0)*=0{\bullet}="b",(2.25,-0.25)*=0{b},(2,2)*=0{\bullet}="c",(2.25,2)*=0{c},(2,1)*=0{\bullet}="a'",(1,0)*=0{\bullet}="c'","a";"a'"**@{-}?!{"c";"c'"}*{\bullet}="x",(1.08,0.92)*=0{x},"a";"b"**@{.},"c";"b"**@{-},"c";"c'"**@{.}\endxy
\]
~

\noindent The definitions of the interval space concepts and notations
follow immediatley afterwards. The proof is given below. For counter-examples
and more examples and history of the concepts, see \cite[sections 1.4, 1.5]{retter_2013}
and the references given there. For a set $X\,,$ $P\left(X\right)$
denotes the power set of $X\,,$ i.e. the set of all subsets of $X\,.$
\begin{thm*}
\ref{sub:interval_transitivity_criterion} (interval-transitivity
criterion) Let $X$ be an interval space. Then the following conditions
are equivalent:
\begin{enumerate}
\item $X$ is interval-transitive.
\item For all $a,\, b,\, c\in X\,,$ $\left[\left\{ a\right\} ,\,\left[b,\, c\right]\right]\subseteq\left[\left[a,\, b\right],\,\left\{ c\right\} \right]\,.$
\item For all $a,\, b,\, c\in X\,,$ $\left[\left\{ a\right\} ,\,\left[b,\, c\right]\right]=\left[\left[a,\, b\right],\,\left\{ c\right\} \right]\,.$
\item $P\left(X\right)$ with the binary operation $\left[\cdot,\,\cdot\right]$
is a semigroup.
\item $P\left(X\right)$ with the binary operation $\left[\cdot,\,\cdot\right]$
is a commutative semigroup.
\item $X$ is interval-convex, and for each convex set $A\,,$ the binary
relation $\left\langle A,\,\cdot,\,\cdot\right\rangle $ is transitive.
\item For all convex sets $A,\, B\,,$ $\left[A,\, B\right]$ is convex.
\item For all $a,\, b,\, c\in X\,,$ $\left[\left[a,\, b\right],\,\left\{ c\right\} \right]$
is convex.
\item For all $a,\, b,\, c\in X\,,$ $\mbox{co}\left(\left\{ a,\, b,\, c\right\} \right)=\left[\left[a,\, b\right],\,\left\{ c\right\} \right]\,.$
\end{enumerate}
\end{thm*}
\noindent Let $X$ be an interval space.

For $A\subseteq X$ and $b,\, c\in X\,,$
\begin{align*}
\left\langle A,\, b,\, c\right\rangle  & :\Longleftrightarrow\mbox{There is an }a\in A\mbox{ such that }\left\langle a,\, b,\, c\right\rangle \,.
\end{align*}

\noindent For $A,\, C\subseteq X$ and $b\in X\,,$
\begin{align*}
\left\langle A,\, b,\, C\right\rangle  & :\Longleftrightarrow\mbox{There are }a\in A\,,\, c\in C\mbox{\,\ such that }\left\langle a,\, b,\, c\right\rangle \,.
\end{align*}

\noindent \begin{flushleft}
For $a,\, c\in X\,,$ the \emph{interval} between $a$ and $c$ is
the set
\par\end{flushleft}

\noindent 
\begin{align*}
\left[a,\, c\right] & :=\left\langle a,\,\cdot,\, c\right\rangle \\
 & =\left\{ x\in X\,|\,\left\langle a,\, x,\, c\right\rangle \right\} \,.
\end{align*}

\noindent \begin{flushleft}
For $A,\, C\subseteq X\,,$ the \emph{interval} between $A$ and $C$
is the set
\par\end{flushleft}

\noindent 
\begin{align*}
\left[A,\, C\right] & :=\left\langle A,\,\cdot,\, C\right\rangle \\
 & =\left\{ x\in X\,|\,\left\langle A,\, x,\, C\right\rangle \right\} \,.
\end{align*}

\noindent A subset $C$ of $X$ is called \emph{convex} iff $\left[C,\, C\right]\subseteq C\,,$
i.e. for all $x,\, y,\, z\in X\,,$ if $\left\langle x,\, y,\, z\right\rangle $
and $x,\, z\in C\,,$ then $y\in C\,.$ 

For $A\subseteq X\,,$ the \emph{convex closure} or \emph{convex hull}
of $A$ in $X$ is the set
\begin{align*}
\mbox{co}\left(A\right) & :=\bigcap\left\{ B\subseteq X|B\supseteq A\mbox{ and }B\mbox{ is convex.}\right\} \,.
\end{align*}

\noindent It is the smallest convex set in $X$ containg $A\,.$

$X$ is called \emph{interval-transitive} iff for for all $a,\, b\in X\,,$
the binary relation $\left\langle \left[a,\, b\right],\,\cdot,\,\cdot\right\rangle $
is transitive. Each interval-transitive interval space is point-transitive.

$X$ is called \emph{interval-convex} iff for all $a,\, b\in X\,,$
$\left[a,\, b\right]$ is convex. The concept of an interval-convex
interval space generalizes the concept of an interval monotone graph
in \cite[1.1.6]{mulder_1980}. The term 'interval-convex' has been
introduced in \cite[section 1.5]{retter_2013}.

An interval space $X$ is called \emph{point-antisymmetric} iff for
each $a\in X\,,$ the binary relation $\left\langle a,\,\cdot,\,\cdot\right\rangle $
is antisymmetric, i.e. for all $x,\, y\in X\setminus\left[a,\, b\right]\,,$
$\left(\left\langle a,\, x,\, y\right\rangle \mbox{ and }\left\langle a,\, y,\, x\right\rangle \right)\Longrightarrow x=y\,.$
A vector space with its vector interval relation is point-antisymmetric.
It is also interval-transitive and satisfies the equivalent conditions
of the following theorem. Condition (\ref{enu:interval_antisymmetry_criterion_1})
is obtained from the point-antisymmetry condition that $\left\langle a,\,\cdot,\,\cdot\right\rangle $
is antisymmetric by replacing the point $a$ by an interval. Condition
(\ref{enu:interval_antisymmetry_criterion_2}) is the interval relation
version of the strict interval relation condition \cite[§1, VIII. Grundsatz]{pasch_1882}.
The definitions of the interval space and closure space concepts follow
immediatley afterwards. The proof is given  below. For counter-examples
and more examples and history of the interval space concepts, see
\cite[sections 1.4, 1.5]{retter_2013} and the references given there.
\begin{thm*}
\ref{sub:interval_antisymmetry_criterion} (interval-antisymmetry
criterion) Let $X$ be an interval-transitive interval space. Then
the following conditions are equivalent:
\begin{enumerate}
\item $X$ is interval-antisymmetric.
\item $X$ is stiff.
\item For each convex set $A\,,$ the binary relation $\left\langle A,\,\cdot,\,\cdot\right\rangle $
is antisymmetric on $X\setminus A\,.$
\item The pair consisting of $X$ and the set of convex sets is an antiexchange
space.
\item The pair consisting of $X$ and the set of convex sets is an antimatroid.
\end{enumerate}
\end{thm*}
\noindent Let $X$ be an interval space.

$X$ is called \emph{interval-antisymmetric} iff for all $a,\, b\in X$
the binary relation $\left\langle \left[a,\, b\right],\,\cdot,\,\cdot\right\rangle $
is antisymmetric on $X\setminus\left[a,\, b\right]\,.$ In general,
there will be no cahnce that this relation is antisymmetric on the
whole of $X\,.$

$X$ is called \emph{stiff} iff for all $a,\, b,\, c,\, d\in X\,,$
$\left(\left\langle a,\, b,\, c\right\rangle \mbox{ and }b\neq c\mbox{ and }\left\langle b,\, c,\, d\right\rangle \right)\Longrightarrow\left\langle a,\, b,\, d\right\rangle \,.$
As noted above, this condition is the interval relation version of
the strict interval relation condition \cite[§1, VIII. Grundsatz]{pasch_1882}.
In \cite[section 3.2]{retter_2013} a stiff interval space has been
called \emph{one-way}.

\[
\xy<1cm,0cm>:(0,0)*=0{\bullet}="a",(0,-0.05)*=0{}="a'",(0,-0.25)*=0{a},(1,0)*=0{\bullet}="b",(1,0.05)*=0{}="b'",(1,-0.25)*=0{b},(2,0)*=0{\bullet}="c",(2,-0.05)*=0{}="c'",(2,-0.25)*=0{c},(4,0)*=0{\bullet}="d",(4,0.05)*=0{}="d'",(4,-0.25)*=0{d},"a'";"c'"**@{-},"b'";"d'"**@{-},"a";"d"**@{.}\endxy
\]

~

Let $X$ be a set. A \emph{closure system} or \emph{Moore family}
on $X$ is a set $C$ of subsets of $X$ such that $X\in C$ and for
each non-empty $D\subseteq C\,,$ $\bigcap D\in C\,.$

A \emph{closure space} is a pair consisting of a set $X$ and a closure
system $C$ on $X\,.$ A set $A\subseteq X$ is called \emph{closed}
iff $A\in C\,.$ When $\left(X,\, O\right)$ is a topological space,
then the pair consisting of $X$ and the set of closed sets in $\left(X,\, O\right)$
is a closure space. When $\left(X,\,\left\langle \cdot,\,\cdot,\,\cdot\right\rangle \right)$
is an interval space, then the pair consisting of $X$ and the set
of convex sets is a closure space. The concept of a closure space
as defined here is slighly more general than in \cite[chapter I, 1.2]{vel_1993},
where it is required that $\emptyset\in C$ and a closure system is
called a protopology. 

A closure space $\left(X,\, C\right)$ is also simply denoted by $X$
when it is clear from the context whether the closure space or only
the set is meant.

Let $\left(X,\, C\right)$ be a closure space.

For $A\subseteq X\,,$ the \emph{closure} of $A$ is the set
\begin{align*}
\mbox{cl}\left(A\right) & :=\bigcap\left\{ B\subseteq X|B\supseteq A\mbox{ and }B\in C\right\} 
\end{align*}

\noindent It is the smallest closed superset of $A\,.$ When $X$
is an interval space and $C$ is the system of convex sets in $X\,,$
then for $A\subseteq X\,,$ the closure of $A$ is the convex closure
of $A\,.$

For $A\subseteq X\,,$ the \emph{entailment relation} of $C$ relative
to $A$ or $A$-entailment relation is the binary relation $\vdash_{A}$
on $X$ defined by
\begin{align*}
x\vdash_{A}y & :\Leftrightarrow y\in\mbox{cl}\left(A\cup\left\{ x\right\} \right)\,.
\end{align*}

\noindent $\left(X,\, C\right)$is called an \emph{antiexchange space}
iff for each closed $A\subseteq X\,,$ one and therefore all of the
following conditions hold, which are equivalent by \cite[Proposition 3.1.1]{retter_2013}:
\begin{itemize}
\item The relation $\vdash_{A}$ is antisymmetric on $X\setminus A\,.$
\item The restriction $\vdash_{A}|\left(X\setminus A\right)$ is a partial
order on $X\setminus A\,.$
\end{itemize}
\noindent $\left(X,\, C\right)$ is called \emph{algebraic} or \emph{combinatorial}
iff for each chain $D\subseteq C\,,$ $\bigcup D\in C\,.$ \cite[chapter I, 1.3]{vel_1993}
states the equivalence of this definition with other well-known definitions.
A combinatorial (i.e. algebraic) closure space is also just called
a \emph{combinatorial space}. When $X$ is an interval space and $C$
is the system of convex sets in $X\,,$ then $\left(X,\, C\right)$
is a combinatorial closure space.

An \emph{antimatroid} (\emph{anti-matroid}) or \emph{Dilworth space}
is a combinatorial exchange space with $\emptyset$ closed. This concept
has been taken from \cite[chapter I, 2.24]{vel_1993}.

\section{Interval-Transitivity Criterion}

\noindent Part (\ref{enu:set_interval_operator_1}) of the following
proposition is cited from \cite[Theorem 2.3]{prenowitz_jantosciak_1979}.
Parts (\ref{enu:set_interval_operator_2}) and (\ref{enu:set_interval_operator_3})
are cited from \cite[Theorem 2.1]{prenowitz_jantosciak_1979}.
\begin{prop}
\label{sub:set_interval_operator} (set interval operator) Let $X$
be an interval space. The binary operation $\left[\cdot,\,\cdot\right]$
on $P\left(X\right)$ has the following properties:
\begin{enumerate}
\item \label{enu:set_interval_operator_1}$\left[\cdot,\,\cdot\right]$
is commutative, i.e. for $A,\, B\subseteq X\,,$ $\left[A,\, B\right]=\left[B,\, A\right]\,.$
\item \label{enu:set_interval_operator_2}For $C\subseteq X\,,$ the unary
operation $\left[\cdot,\, C\right]$ is increasing, i.e. for $A,\, B\subseteq X\,,$
$A\subseteq B\Longrightarrow\left[A,\, C\right]\subseteq\left[B,\, C\right]\,.$
\item \label{enu:set_interval_operator_3}For $C\subseteq X\,,$ the unary
operation $\left[C,\,\cdot\right]$ is increasing, i.e. for $A,\, B\subseteq X\,,$
$A\subseteq B\Longrightarrow\left[C,\, A\right]\subseteq\left[C,\, B\right]\,.$
\item \label{enu:set_interval_operator_4}The binary operation $\left[\cdot,\,\cdot\right]$
is increasing, i.e. for $A,\, B,\, C,\, D\subseteq X\,,$ $A\subseteq B\mbox{ and }C\subseteq D\Longrightarrow\left[A,\, C\right]\subseteq\left[B,\, D\right]\,.$
\end{enumerate}
\end{prop}
\begin{proof}
~
\begin{enumerate}
\item \cite[Theorem 2.3]{prenowitz_jantosciak_1979}
\item \cite[Theorem 2.1]{prenowitz_jantosciak_1979}
\item \cite[Theorem 2.1]{prenowitz_jantosciak_1979}
\item follows from (\ref{enu:set_interval_operator_3}) and (\ref{enu:set_interval_operator_4})
\end{enumerate}
\end{proof}
Let $X$ be an interval space and $a,\, b\in X\,.$ If the binary
relation $\left\langle \left[a,\, b\right],\,\cdot,\,\cdot\right\rangle $
is transitive, then for $x,\, y,\, c\in X\,,$
\begin{align*}
\left\langle \left[a,\, b\right],\, x,\, y\right\rangle \mbox{ and }\left\langle \left[a,\, b\right],\, y,\, c\right\rangle  & \Longrightarrow\left\langle \left[a,\, b\right],\, x,\, c\right\rangle \,.
\end{align*}

\noindent Substituting $a\in\left[a,\, b\right]$ and $b\in\left[a,\, b\right]\,,$
\begin{align*}
\left\langle a,\, x,\, y\right\rangle \mbox{ and }\left\langle b,\, y,\, c\right\rangle  & \Longrightarrow\left\langle \left[a,\, b\right],\, x,\, c\right\rangle \,,
\end{align*}

\noindent i.e.
\begin{align*}
\left\langle a,\, x,\, y\right\rangle \mbox{ and }y\in\left[b,\, c\right] & \Longrightarrow\left\langle \left[a,\, b\right],\, x,\, c\right\rangle \,.
\end{align*}

\noindent Consequently, for $x,\, c\in X\,,$
\begin{align*}
\left\langle a,\, x,\,\left[b,\, c\right]\right\rangle  & \Longrightarrow\left\langle \left[a,\, b\right],\, x,\, c\right\rangle \,,
\end{align*}
i.e. for $c\in X\,,$ 
\begin{align*}
\left[\left\{ a\right\} ,\,\left[b,\, c\right]\right]\subseteq & \left[\left[a,\, b\right],\,\left\{ c\right\} \right]\,.
\end{align*}

\noindent The last condition says: For $x\in X\,,$ if $x$ is between
$a$ and $\left[b,\, c\right]\,,$ then $x$ is also between $\left[a,\, b\right]$
and $c\,:$

~

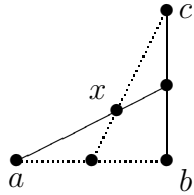
\noindent 
\[
\xy<1cm,0cm>:(0,0)*=0{\bullet}="a",(0,-0.25)*=0{a},(2,0)*=0{\bullet}="b",(2.25,-0.25)*=0{b},(2,2)*=0{\bullet}="c",(2.25,2)*=0{c},(2,1)*=0{\bullet}="a'",(1,0)*=0{\bullet}="c'","a";"a'"**@{-}?!{"c";"c'"}*{\bullet}="x",(1.08,0.92)*=0{x},"a";"b"**@{.},"c";"b"**@{-},"c";"c'"**@{.}\endxy
\]

~

\noindent Summarizing,
\begin{prop}
\label{sub:interval_spaces_transitive_from_a_base_interval} (interval
spaces transitive from a base-interval) Let $X$ be an interval space
and $a,\, b\in X\,.$ If the binary relation $\left\langle \left[a,\, b\right],\,\cdot,\,\cdot\right\rangle $
is transitive, then for $c\in X\,,$ $\left[a,\,\left[b,\, c\right]\right]\subseteq\left[\left[a,\, b\right],\, c\right]\,.$
\end{prop}
\noindent Condition (\ref{enu:interval_transitivity_criterion_2})
in the following theorem is the interval relation version of the strict
interval relation condition \cite[§10, Assioma XIII]{peano_1889}.
In \cite[chapter I, 4.9]{vel_1993} it has been called the \emph{Peano
Property}. In \cite[section 1.5]{retter_2013}, its equivalence with
conditions (\ref{enu:interval_transitivity_criterion_8}) and (\ref{enu:interval_transitivity_criterion_9})
has been stated, and accordingly, $X$ has been called \emph{triangle-convex}
iff one and therefore each of these three equivalent conditions is
satisfied. The implication (\ref{enu:interval_transitivity_criterion_3})
$\Rightarrow$ (\ref{enu:interval_transitivity_criterion_4}) has
been proved in \cite[Theorem 2.4]{prenowitz_jantosciak_1979}. In
\cite[chapter I, 4.10 (2)]{vel_1993} it has been shown that (\ref{enu:interval_transitivity_criterion_2})
implies interval-convexity in (\ref{enu:interval_transitivity_criterion_6}).
In \cite[Theorem 4.45 (a)]{prenowitz_jantosciak_1979} it has been
demonstrated that (\ref{enu:interval_transitivity_criterion_3}) implies
the strict interval relation version of transitivity in (\ref{enu:interval_transitivity_criterion_6}).
The implication (\ref{enu:interval_transitivity_criterion_2}) $\Rightarrow$
(\ref{enu:interval_transitivity_criterion_7}) has been proved in
\cite[Theorem 2.12]{prenowitz_jantosciak_1979}.
\begin{thm}
\label{sub:interval_transitivity_criterion} (interval-transitivity
criterion) Let $X$ be an interval space. Then the following conditions
are equivalent:
\begin{enumerate}
\item \label{enu:interval_transitivity_criterion_1}$X$ is interval-transitive.
\item \label{enu:interval_transitivity_criterion_2}For all $a,\, b,\, c\in X\,,$
$\left[\left\{ a\right\} ,\,\left[b,\, c\right]\right]\subseteq\left[\left[a,\, b\right],\,\left\{ c\right\} \right]\,.$
\item \label{enu:interval_transitivity_criterion_3}For all $a,\, b,\, c\in X\,,$
$\left[\left\{ a\right\} ,\,\left[b,\, c\right]\right]=\left[\left[a,\, b\right],\,\left\{ c\right\} \right]\,.$
\item \label{enu:interval_transitivity_criterion_4}$P\left(X\right)$ with
the binary operation $\left[\cdot,\,\cdot\right]$ is a semigroup.
\item \label{enu:interval_transitivity_criterion_5}$P\left(X\right)$ with
the binary operation $\left[\cdot,\,\cdot\right]$ is a commutative
semigroup.
\item \label{enu:interval_transitivity_criterion_6}$X$ is interval-convex,
and for each convex set $A\,,$ the binary relation $\left\langle A,\,\cdot,\,\cdot\right\rangle $
is transitive.
\item \label{enu:interval_transitivity_criterion_7}For all convex sets
$A,\, B\,,$ $\left[A,\, B\right]$ is convex.
\item \label{enu:interval_transitivity_criterion_8}For all $a,\, b,\, c\in X\,,$
$\left[\left[a,\, b\right],\,\left\{ c\right\} \right]$ is convex.
\item \label{enu:interval_transitivity_criterion_9}For all $a,\, b,\, c\in X\,,$
$\mbox{co}\left(\left\{ a,\, b,\, c\right\} \right)=\left[\left[a,\, b\right],\,\left\{ c\right\} \right]\,.$
\end{enumerate}
\end{thm}
\begin{proof}
Step 1. (\ref{enu:interval_transitivity_criterion_1}) $\Rightarrow$
(\ref{enu:interval_transitivity_criterion_2}) follows by \ref{sub:interval_spaces_transitive_from_a_base_interval}
(interval spaces transitive from a base-interval).

Step 2. (\ref{enu:interval_transitivity_criterion_2}) $\Rightarrow$
(\ref{enu:interval_transitivity_criterion_3}). For $a,\, b,\, c\in X$
it remains to be proved that $=\left[\left[a,\, b\right],\, c\right]\subseteq\left[a,\,\left[b,\, c\right]\right]\,.$
The assumption (\ref{enu:interval_transitivity_criterion_2}) implies
by \ref{sub:set_interval_operator}(\ref{enu:set_interval_operator_1})
(set interval operator):
\begin{align*}
 & \left[\left[a,\, b\right],\,\left\{ c\right\} \right]\\
= & \left[\left\{ c\right\} ,\,\left[b,\, a\right]\right]\\
\subseteq & \left[\left[c,\, b\right],\,\left\{ a\right\} \right]\\
= & \left[\left\{ a\right\} ,\,\left[b,\, c\right]\right]\,.
\end{align*}

Step 3. (\ref{enu:interval_transitivity_criterion_3}) $\Rightarrow$
(\ref{enu:interval_transitivity_criterion_4}). \cite[Theorem 2.4]{prenowitz_jantosciak_1979}

Step 4. (\ref{enu:interval_transitivity_criterion_4}) $\Rightarrow$
(\ref{enu:interval_transitivity_criterion_5}) follows by \ref{sub:set_interval_operator}(\ref{enu:set_interval_operator_1})
(set interval operator).

Step 5. (\ref{enu:interval_transitivity_criterion_5}) $\Rightarrow$
(\ref{enu:interval_transitivity_criterion_6}). 

Step 5.1. Proof that $X$ is interval-convex, i.e. for $a,\, b\in X\,,$
$\left[\left[a,\, b\right],\,\left[a,\, b\right]\right]\subseteq\left[a,\, b\right]\,.$
The assumption (\ref{enu:interval_transitivity_criterion_5}) implies
by generalized associativity and commutativity: 
\begin{align*}
 & \left[\left[a,\, b\right],\,\left[a,\, b\right]\right]\\
= & \left[\left[\left\{ a\right\} ,\,\left\{ b\right\} \right],\,\left[\left\{ a\right\} ,\,\left\{ b\right\} \right]\right]\\
= & \left[\left[\left\{ a\right\} ,\,\left\{ a\right\} \right],\,\left[\left\{ b\right\} ,\,\left\{ b\right\} \right]\right]\\
= & \left[\left\{ a\right\} ,\,\left\{ b\right\} \right]\\
= & \left[a,\, b\right]\,.
\end{align*}

\noindent Step 5.2. Proof that for each convex set $A\,,$ the binary
relation $\left\langle A,\,\cdot,\,\cdot\right\rangle $ is transitive,
i.e. for $x,\, y,\, z\in X\,,$ $\left\langle A,\, x,\, y\right\rangle $
and $\left\langle A,\, y,\, z\right\rangle $ implies $\left\langle A,\, x,\, z\right\rangle \,,$
i.e. for $y,\, z\in X\,,$ $y\in\left[A,\,\left\{ z\right\} \right]$
implies $\left[A,\,\left\{ y\right\} \right]\subseteq\left[A,\,\left\{ z\right\} \right]\,.$
The assumption that $A$ is convex says:
\begin{align}
\left[A,\, A\right] & \subseteq A\,.\label{eq:interval_transitivity_criterion_1}
\end{align}
From the assumptions $y\in\left[A,\,\left\{ z\right\} \right]\,,$
i.e. $\left\{ y\right\} \subseteq\left[A,\,\left\{ z\right\} \right]\,,$
and (\ref{enu:interval_transitivity_criterion_5}) and (\ref{eq:interval_transitivity_criterion_1})
it follows by \ref{sub:set_interval_operator}(\ref{enu:set_interval_operator_2})
and (\ref{enu:set_interval_operator_3}) (set interval operator):
\begin{align*}
 & \left[A,\,\left\{ y\right\} \right]\\
\subseteq & \left[A,\,\left[A,\,\left\{ z\right\} \right]\right]\\
= & \left[\left[A,\, A\right],\,\left\{ z\right\} \right]\\
\subseteq & \left[A,\,\left\{ z\right\} \right]\,.
\end{align*}

Step 6. (\ref{enu:interval_transitivity_criterion_6}) $\Rightarrow$
(\ref{enu:interval_transitivity_criterion_1}). The assumption that
$X$ is interval-convex entails:
\begin{align}
 & \left[a,\, b\right]\mbox{ is convex.}\label{eq:interval_transitivity_criterion_2}
\end{align}

\noindent From (\ref{eq:interval_transitivity_criterion_2}) and the
assumption that for each convex set $A$ the binary relation $\left\langle A,\,\cdot,\,\cdot\right\rangle $
is transitive it follows that $\left\langle \left[a,\, b\right],\,\cdot,\,\cdot\right\rangle $
is transitive.

Step 7. (\ref{enu:interval_transitivity_criterion_5}) $\Rightarrow$
(\ref{enu:interval_transitivity_criterion_7}). The assumption that
$A,\, B$ are convex says:
\begin{align}
\left[A,\, A\right] & \subseteq A\,,\label{eq:interval_transitivity_criterion_3}\\
\left[B,\, B\right] & \subseteq B\,.\label{eq:interval_transitivity_criterion_4}
\end{align}

\noindent The assumption (\ref{enu:interval_transitivity_criterion_5}),
(\ref{eq:interval_transitivity_criterion_3}) and (\ref{eq:interval_transitivity_criterion_4})
imply by by generalized associativity and commutativity: and \ref{sub:set_interval_operator}(\ref{enu:set_interval_operator_4})
(set interval operator): 
\begin{align*}
 & \left[\left[A,\, B\right],\,\left[A,\, B\right]\right]\\
= & \left[\left[A,\, A\right],\,\left[B,\, B\right]\right]\\
\subseteq & \left[A,\, B\right]\,.
\end{align*}

Step 8. (\ref{enu:interval_transitivity_criterion_7}) $\Rightarrow$
(\ref{enu:interval_transitivity_criterion_8}).
\begin{align}
 & \left\{ a\right\} ,\,\left\{ b\right\} ,\,\left\{ c\right\} \mbox{ are convex.}\label{eq:interval_transitivity_criterion_6}
\end{align}
From (\ref{eq:interval_transitivity_criterion_6}) and the assumption
(\ref{enu:interval_transitivity_criterion_7}) it follows that $\left[\left\{ a\right\} ,\,\left\{ b\right\} \right]$
is convex, i.e. 
\begin{align}
 & \left[a,\, b\right]\mbox{ is convex.}\label{eq:interval_transitivity_criterion_5}
\end{align}

\noindent From (\ref{eq:interval_transitivity_criterion_5}) and the
assumption (\ref{enu:interval_transitivity_criterion_7}) imply that
$\left[\left[a,\, b\right],\, c\right]$ is convex.

Step 9. (\ref{enu:interval_transitivity_criterion_8}) $\Rightarrow$
(\ref{enu:interval_transitivity_criterion_9}). The assumption (\ref{enu:interval_transitivity_criterion_8})
implies that it suffices to prove that for $C$ a convex set, $C\supseteq\left\{ a,\, b,\, c\right\} $
iff $C\supseteq\left[\left[a,\, b\right],\, c\right]\,.$ From the
assumption that $C$ is convex it follows: 
\begin{align*}
 & C\supseteq\left\{ a,\, b,\, c\right\} \\
\Longleftrightarrow & \left(C\supseteq\left\{ a,\, b\right\} \mbox{ and }C\supseteq\left\{ c\right\} \right)\\
\Longleftrightarrow & \left(C\supseteq\left[a,\, b\right]\mbox{ and }C\supseteq\left\{ c\right\} \right)\\
\Longleftrightarrow & C\supseteq\left[\left[a,\, b\right],\, c\right]
\end{align*}

Step 10. (\ref{enu:interval_transitivity_criterion_9}) $\Rightarrow$
(\ref{enu:interval_transitivity_criterion_3}). The assumption (\ref{enu:interval_transitivity_criterion_9})
implies by \ref{sub:set_interval_operator}(\ref{enu:set_interval_operator_1})
(set interval operator):
\begin{align*}
 & \left[\left\{ a\right\} ,\,\left[b,\, c\right]\right]\\
= & \left[\left[b,\, c\right],\,\left\{ a\right\} \right]\\
= & \mbox{co}\left(\left\{ b,\, c,\, a\right\} \right)\\
= & \mbox{co}\left(\left\{ a,\, b,\, c\right\} \right)\\
= & \left[\left[a,\, b\right],\,\left\{ c\right\} \right]\,.
\end{align*}

\end{proof}

\section{Interval-Antisymmetry Crtierion}

\noindent Let $X$ be point-transitive interval space and $a,\, d\in X\,.$
If the binary relation $\left\langle \left[a,\, d\right],\,\cdot,\,\cdot\right\rangle $
is antisymmetric on $X\setminus\left[a,\, d\right]\,,$ then for $b,\, c\in X\,,$
\begin{align*}
 & \left(c,\, b\notin\left[a,\, d\right]\mbox{ and }\left\langle \left[a,\, d\right],\, b,\, c\right\rangle \mbox{ and }\left\langle \left[a,\, d\right],\, c,\, b\right\rangle \right)\Longrightarrow b=c\,.
\end{align*}

\noindent Substituting $a\in\left[a,\, d\right]$ in the second and
and $d\in\left[a,\, d\right]$ in the third condition,
\begin{align*}
 & \left(c,\, b\notin\left[a,\, d\right]\mbox{ and }\left\langle a,\, b,\, c\right\rangle \mbox{ and }\left\langle d,\, c,\, b\right\rangle \right)\Longrightarrow b=c\,.
\end{align*}

\noindent Equivalently,
\begin{align*}
 & \left(\left\langle a,\, b,\, c\right\rangle \mbox{ and }\left\langle d,\, c,\, b\right\rangle \mbox{ and }b\neq c\right)\Longrightarrow\left(c\in\left[a,\, d\right]\mbox{ or }b\in\left[a,\, d\right]\right)\,.
\end{align*}

\noindent Rewriting,
\begin{align*}
 & \left(\left\langle a,\, b,\, c\right\rangle \mbox{ and }b\neq c\mbox{ and }\left\langle b,\, c,\, d\right\rangle \right)\Longrightarrow\left(\left\langle a,\, c,\, d\right\rangle \mbox{ or }\left\langle a,\, b,\, d\right\rangle \right)\,.
\end{align*}

\noindent The condition $\left\langle a,\, b,\, c\right\rangle $
on the left, the first possibility $\left\langle a,\, c,\, d\right\rangle $
on the right and the assumption that $X$ is point-transitive imply
the second possibility $\left\langle b,\, c,\, d\right\rangle $ on
the right. Consequently,
\begin{align*}
 & \left(\left\langle a,\, b,\, c\right\rangle \mbox{ and }b\neq c\mbox{ and }\left\langle b,\, c,\, d\right\rangle \right)\Longrightarrow\left\langle a,\, b,\, d\right\rangle \,.
\end{align*}

\noindent Summarizing,
\begin{prop}
\label{sub:interval_spaces_antisymmetric_from_a_base_interval} (interval
spaces antisymmetric from a base-interval) Let $X$ be a point-transitive
interval space and $a,\, d\in X\,.$ If the binary relation $\left\langle \left[a,\, d\right],\,\cdot,\,\cdot\right\rangle $
is antisymmetric on $X\setminus\left[a,\, d\right]\,,$ then for $b,\, c\in X\,,$
$\left(\left\langle a,\, b,\, c\right\rangle \mbox{ and }b\neq c\mbox{ and }\left\langle b,\, c,\, d\right\rangle \right)\Longrightarrow\left\langle a,\, b,\, d\right\rangle \,.$
\end{prop}
\noindent The following proposition is similar to \cite[chapter II, proposition 10]{coppel_1998}.
\begin{prop}
\label{sub:stiff_interval_spaces} (stiff interval spaces) Let $X$
be a stiff interval space. Then for each convex set $A\,,$ the binary
relation $\left\langle A,\,\cdot,\,\cdot\right\rangle $ is antisymmetric
on $X\setminus A\,.$\end{prop}
\begin{proof}
It is to be proved that $b,\, c\in X\setminus A\,,$ $\left\langle A,\, b,\, c\right\rangle $
and $\left\langle A,\, c,\, b\right\rangle $ implies $b=c\,.$ The
assumptions $\left\langle A,\, b,\, c\right\rangle $ and $\left\langle A,\, c,\, b\right\rangle $
say that there are
\begin{align}
a,\, d & \in A\label{eq:stiff_interval_spaces_1}
\end{align}

\noindent such that
\begin{align}
 & \left\langle a,\, b,\, c\right\rangle \label{eq:stiff_interval_spaces_2}
\end{align}
and $\left\langle d,\, c,\, b\right\rangle \,,$ i.e.
\begin{align}
 & \left\langle b,\, c,\, d\right\rangle \,.\label{eq:stiff_interval_spaces_3}
\end{align}
Seeking a contradiction, suppose
\begin{align}
b & \neq c\,.\label{eq:stiff_interval_spaces_4}
\end{align}

\noindent (\ref{eq:stiff_interval_spaces_2}), (\ref{eq:stiff_interval_spaces_4}),
(\ref{eq:stiff_interval_spaces_3}) and the assumption (\ref{enu:interval_antisymmetry_criterion_2})
imply:
\begin{align}
 & \left\langle a,\, b,\, d\right\rangle \,.\label{eq:stiff_interval_spaces_5}
\end{align}

\noindent From (\ref{eq:stiff_interval_spaces_1}), (\ref{eq:stiff_interval_spaces_5})
and the assumption that $A$ is convex it follows that $b\in A\,,$
contradicting the assumption that $b\in X\setminus A\,.$\end{proof}
\begin{prop}
\label{sub:interval_transitive_interval_spaces} (interval-transitve
interval spaces) Let $X$ be an interval-transitive interval space
and $A$ a convex set. Then the relative entailment relation $\vdash_{A}$
is the reverse relation of the binary relation $\left\langle A,\,\cdot,\,\cdot\right\rangle \,.$\end{prop}
\begin{proof}
\noindent It is to be proved that for $b,\, c\in X\,,$ $c\vdash_{A}b$
iff $\left\langle A,\, b,\, c\right\rangle \,.$
\begin{align}
 & \left\{ c\right\} \mbox{ is convex.}\label{eq:interval_transitive_interval_spaces_1}
\end{align}
From (\ref{eq:interval_transitive_interval_spaces_1}) and the assumptions
that $A$ is convex and $X$ is interval-transitive it follows by
\ref{sub:interval_transitivity_criterion} (interval-transitivity
criterion) that $\left[A,\,\left\{ c\right\} \right]$ is convex.
Therefore, $\mbox{co}\left(A\cup\left\{ c\right\} \right)=\left[A,\,\left\{ c\right\} \right]\,.$
Consequently, the following equivalences hold:
\begin{align*}
 & c\vdash_{A}b\\
\Longleftrightarrow & b\in\mbox{co}\left(A\cup\left\{ c\right\} \right)\\
\Longleftrightarrow & b\in\left[A,\,\left\{ c\right\} \right]\\
\Longleftrightarrow & \left\langle A,\, b,\, c\right\rangle \,.
\end{align*}

\end{proof}
\pagebreak{}

\noindent The following proposition is a particular case of a more
general principle for relational structures.
\begin{prop}
\label{sub:interval_spaces_are_combinatorial_spaces} (interval spaces
are combinatorial spaces) Let $X$ be an interval space. Then the
closure space consisting of $X$ and the set of convex sets is combinatorial.\end{prop}
\begin{proof}
For a chain $D$ of convex sets is to be proved that $\bigcup D$
is convex, i.e. for $a,\, b,\, c\in X\,,$ if $a,\, c\in\bigcup D$
and $\left\langle a,\, b,\, c\right\rangle \,,$ then $b\in\bigcup D\,.$
The assumption that $a,\, c\in\bigcup D$ says that there are $A,\, C\in D$
such that
\begin{align}
 & a\in A\,,\label{eq:interval_spaces_are_combinatorial_spaces_1}\\
 & c\in C\,.\label{eq:interval_spaces_are_combinatorial_spaces_2}
\end{align}
From the assumptions that $D$ is a chain and $A,\, C\in D$ it follows
that $A\subseteq C$ or $C\subseteq A\,.$ Suppose without loss of
generality that
\begin{align}
A & \subseteq C\,.\label{eq:interval_spaces_are_combinatorial_spaces_3}
\end{align}
(\ref{eq:interval_spaces_are_combinatorial_spaces_1}) and (\ref{eq:interval_spaces_are_combinatorial_spaces_3})
imply
\begin{align}
a & \in C\,.\label{eq:interval_spaces_are_combinatorial_spaces_4}
\end{align}
From (\ref{eq:interval_spaces_are_combinatorial_spaces_4}), (\ref{eq:interval_spaces_are_combinatorial_spaces_2})
and the assumptions that $\left\langle a,\, b,\, c\right\rangle $
and $C$ is convex it follows:
\begin{align}
b & \in C\,.\label{eq:interval_spaces_are_combinatorial_spaces_5}
\end{align}
(\ref{eq:interval_spaces_are_combinatorial_spaces_5}) and the assumption
that $C\in D$ imply that $b\in\bigcup D\,.$\end{proof}
\begin{thm}
\label{sub:interval_antisymmetry_criterion} (interval-antisymmetry
criterion) Let $X$ be an interval-transitive interval space. Then
the following conditions are equivalent:
\begin{enumerate}
\item \label{enu:interval_antisymmetry_criterion_1}$X$ is interval-antisymmetric.
\item \label{enu:interval_antisymmetry_criterion_2}$X$ is stiff.
\item \label{enu:interval_antisymmetry_criterion_3}For each convex set
$A\,,$ the binary relation $\left\langle A,\,\cdot,\,\cdot\right\rangle $
is antisymmetric on $X\setminus A\,.$
\item \label{enu:interval_antisymmetry_criterion_4}The pair consisting
of $X$ and the set of convex sets is an antiexchange space.
\item \label{enu:interval_antisymmetry_criterion_5}The pair consisting
of $X$ and the set of convex sets is an antimatroid.
\end{enumerate}
\end{thm}
\begin{proof}
Step 1. (\ref{enu:interval_antisymmetry_criterion_1}) $\Rightarrow$
(\ref{enu:interval_antisymmetry_criterion_2}). From the assumption
that $X$ is interval-transitive it follows:
\begin{align}
 & X\mbox{ is point-transitive.}\label{eq:interval_antisymmetry_criterion_1}
\end{align}
From (\ref{eq:interval_antisymmetry_criterion_1}) and the assumption
(\ref{enu:interval_antisymmetry_criterion_1}) it follows \ref{sub:interval_spaces_antisymmetric_from_a_base_interval}
(interval spaces antisymmetric from a base-interval) that $X$ is
stiff.

Step 2. (\ref{enu:interval_antisymmetry_criterion_2}) $\Rightarrow$
(\ref{enu:interval_antisymmetry_criterion_3}) follows from \ref{sub:stiff_interval_spaces}
(stiff interval spaces).

Step 3. (\ref{enu:interval_antisymmetry_criterion_3}) $\Rightarrow$
(\ref{enu:interval_antisymmetry_criterion_1}). For $a,\, d\in X$
it is to be proved that the binary relation $\left\langle \left[a,\, d\right],\,\cdot,\,\cdot\right\rangle $
is antisymmetric on $X\setminus\left[a,\, d\right]\,.$ From the assumption
that $X$ is interval-transitive it follows by \ref{sub:interval_transitivity_criterion}
(interval-transitivity criterion) that $X$ is interval-convex. In
particular,
\begin{align}
 & \left[a,\, d\right]\mbox{ is convex.}\label{eq:interval_antisymmetry_criterion_2}
\end{align}

\noindent (\ref{eq:interval_antisymmetry_criterion_2}) and the assumption
(\ref{enu:interval_antisymmetry_criterion_3}) imply that the binary
relation $\left\langle \left[a,\, d\right],\,\cdot,\,\cdot\right\rangle $
is antisymmetric on $X\setminus\left[a,\, d\right]\,.$

Step 4. (\ref{enu:interval_antisymmetry_criterion_3}) $\Leftrightarrow$
(\ref{enu:interval_antisymmetry_criterion_4}). It is to be proved
iff that for each convex set $A\,,$ the binary relation $\left\langle A,\,\cdot,\,\cdot\right\rangle $
is antisymmetric on $X\setminus A\,.$ for each convex set $A\,,$
the relative entailment relation $\vdash_{A}$ is antisymmetric on
$X\setminus A\,.$ Antismmetry being preserved under passing to the
reverse relation, it suffices to prove that for each convex set $A\,,$
the relation $\vdash_{A}$ is the reverse relation of the relation
$\left\langle A,\,\cdot,\,\cdot\right\rangle \,.$ This claim follows
by \ref{sub:interval_transitive_interval_spaces} (interval-transitve
interval spaces) from the assumption that $X$ is interval-transitive.

Step 5. (\ref{enu:interval_antisymmetry_criterion_4}) $\Leftrightarrow$
(\ref{enu:interval_antisymmetry_criterion_5}) follows by \ref{sub:interval_spaces_are_combinatorial_spaces}
(interval spaces are combinatorial spaces) because $\emptyset$ is
convex.
\end{proof}

\section{Conclusion}

Two well-known geometry axioms have been generated from the poset
axioms of transitivity and antisymmetry, passing from a base-point
to a base-interval in the following sense: By \ref{sub:interval_transitivity_criterion}
(interval-transitivity criterion) , interval-transitivity is equivalent
to the axiom \cite[§10, Assioma XIII]{peano_1889}. By \ref{sub:interval_antisymmetry_criterion}
(interval-antisymmetry criterion), assuming interval-transitivity,
interval-antisymmetry is equivalent to the axiom \cite[§1, VIII. Grundsatz]{pasch_1882}.
Beyond the conditions defining an interval space, these are the first
two axioms in the incremental buildup of order geoemty as developed
in \cite{coppel_1998}. Sticking to a base-point, but passing from
the interval relation to the incidence relation, which is the symmetrized
interval relation, and passing from poset axioms to the axioms for
an equivalence relation, this theme of generating geometry axioms
is continued in \cite[Theorem 3.3.2]{retter_2013}: Assuming the two
axioms above and one further axiom, incidence-transitivity is equivalent
to the conjunction of the next two axioms in the incremental builup
in \cite{coppel_1998}, \cite[(5.2)]{sholander_1952} and \cite[§1, VII. Grundsatz]{pasch_1882}.
These results together reinforce the choice of geometric axioms as
well as the order of the above-mentioned incremental buildup. Research
is under way on the question how many of the other axioms of order
geometry fit into this scheme.

\end{document}